\documentclass[12 pt]{article}
\usepackage{amsmath, amssymb}
\usepackage{amsthm, enumerate}
\usepackage{verbatim}

\newcommand\eps{\epsilon}

\newcommand\la{\lambda}

\newcommand\om{\omega}

\renewcommand\th{\theta} 

\newcommand\cx{{\mathbb C}}

\newcommand\ints{{\mathbb Z}}
\newcommand\re{{\mathbb R}}
\newcommand\rats{{\mathbb Q}}

\newcommand\pmat[1]{\begin{pmatrix} #1 \end{pmatrix}}

\newcommand\one{\textbf{1}}

\DeclareMathOperator{\tr}{tr}

\newcommand\bra[1]{\langle#1|}
\newcommand\ket[1]{|#1\rangle}
\newcommand\braket[2]{\langle#1|#2\rangle}

\renewcommand\Re{\mathrm{Re}}
\renewcommand\Im{\mathrm{Im}}
\newcommand\abs[1]{\left|#1\right|}

\newtheoremstyle{plainsl}
	{\topsep}
	{\topsep}
	{\slshape} 
	{}
	{\normalfont\bfseries}
	{.}
	{ }
	{}

\swapnumbers

\theoremstyle{plainsl}
\newtheorem{theorem}{Theorem}[section]

\newtheorem{lemma}[theorem]{Lemma}
\newtheorem{corollary}[theorem]{Corollary}

\theoremstyle{definition}

\renewcommand\proof{\noindent\textsl{Proof. }}
\newcommand\sqr[2]{{\vbox{\hrule height.#2pt
    \hbox{\vrule width.#2pt height#1pt \kern#1pt
        \vrule width.#2pt}\hrule height.#2pt}}}
\renewcommand\qed{%
	\ifmmode\eqno\sqr53
	\else\nolinebreak\ \hfill\sqr53\medbreak\fi}

\numberwithin{equation}{section}

\title{Uniform Mixing and Association Schemes}
\author{Chris Godsil, Natalie Mullin, Aidan Roy}

\begin{document}
\maketitle

\abstract{
We consider continuous-time quantum walks on distance-regular graphs. 
Using results about the existence of complex Hadamard matrices in association schemes, 
we determine which of these graphs have quantum walks that admit uniform mixing.

First we apply a result due to Chan to show that the only strongly regular graphs that admit instantaneous uniform mixing are the Paley graph of order nine and certain graphs corresponding to regular symmetric Hadamard matrices with constant diagonal. Next we prove that if uniform mixing occurs on a bipartite graph $X$ with $n$ vertices, then $n$ is divisible by four. We
also prove that if $X$ is bipartite and regular, then $n$ is the sum of two integer squares.
Our work on bipartite graphs implies that uniform mixing does not occur on $C_{2m}$ for $m \geq 3$. Using a result of Haagerup, we show that uniform mixing does not occur on $C_p$ for any prime $p$ such that $p \geq 5$. In contrast to this result, we see that $\epsilon$-uniform mixing occurs on $C_p$ for all primes $p$.}

\section{Introduction}

Quantum walks are a quantum analogue of a random walk on a graph, and have recently been the
subject of much investigation. For recent survey from a mathematician's viewpoint,
see \cite{godsil12}. A quantum walk can behave quite differently from a classical walk: for example Childs et al.~\cite{childs03}
found a graph in which the time to propagate from one node to another was sped up
exponentially compared to any classical walk.

We describe the physical motivation for continuous-time quantum walks. Let $X$ be
a graph with Laplacian matrix $L$, and let $p(t) = (p_1(t),\ldots,p_n(t))$ be a probability
distribution over the vertices $\{1,\ldots,n\}$ of $X$, for each time $t$. Then $p(t)$ is a
continuous-time random walk on $X$ if it satisfies the differential equation 
\[
\frac{\mathrm{d}p(t)}{\mathrm{d}t} = -L p(t). 
\] 
This is a continuous version of a
discrete-time random walk, in which the Laplacian matrix is the transition matrix for a
Markov process on the vertices. The solution to the differential equation is 
\[ 
	p(t) = e^{-Lt}p(0). 
\] 
Since $L$ has rows which sum to $0$, $e^{-Lt}$ has rows which sum to $1$ and
is therefore a stochastic process.

In a continuous-time quantum walk, instead of a probability distribution $p(t)$ we use a
quantum state $\ket{\psi(t)}$ defined on a state space with orthonormal basis $\{\ket{1},
\ldots, \ket{n} \}$. Instead of a stochastic process, the walk is governed by a unitary
evolution. From Schr\"odinger's equation for the evolution of a quantum system
\[
	i \frac{\mathrm{d}\ket{\psi(t)}}{\mathrm{d}t } = L \ket{\psi(t)},
\]
we have the solution
\[
	\ket{\psi(t)} = e^{itL}\ket{\psi(0)}.
\]
Since the Laplacian $L$ is Hermitian, $e^{iLt}$ is unitary. Assume our quantum walk begins at
some vertex $j$, so $\ket{\psi(0)} = \ket{j}$. If at some time $t$ we measure the state
$\ket{\psi(t)}$ in the vertex basis $\{\ket{1}, \ldots, \ket{n} \}$, then the probability of
outcome $\ket{k}$ is
\[
	\abs{\braket{k}{\psi(t)}}^2 = \abs{\bra{k}e^{iLt}\ket{j}}^2 =  \abs{(e^{iLt})_{kj}}^2.
\]
If this probability distribution is uniform over all vertices $k$, and this is true for all
initial states $\ket{j}$, then the graph is said to have uniform mixing.

In fact we will define the unitary evolution using the adjacency matrix rather than the
Laplacian. When $X$ is regular, the resulting states differ only by an overall phase factor,
and this does not affect uniform mixing.

In this paper we focus on the question of which graphs
have uniform mixing: can a quantum walk beginning at a single vertex result in a state whose
probability distribution is uniform over all vertices? If this is possible at some small time
$t$, then the behaviour is quite different from a classical walk, which typically approaches
the uniform distribution as $t \to \infty$.

Moore and Russell \cite{moore02} in 2001 gave the first example of uniform mixing, showing
that the $n$-cube has uniform mixing at time $t = \pi/4$. Since then a large body of work
has been produced to study this problem. We summarize the conclusions. The following
graphs do admit uniform mixing:
\begin{enumerate}[(a)]
	\item 
	The complete graphs $K_2$, $K_3$, and $K_4$ \cite{ahmadi03}.
	\item
	The Hamming graphs $H(n,q)$, $q \in \{2,3,4\}$ \cite{moore02} \cite{carlson07}.
	\item
	The Cartesian product of any two graphs that admit uniform mixing at the same
	time.
\end{enumerate}
Note that since the Hamming graph $H(n,q)$ is the $n$-th Cartesian power of $K_q$,
the results in (b) follow from the results in (a). A number of Cayley graphs for $\ints_2^d$ admit uniform
mixing, but no classification is yet known. Some work in this direction appears
in \cite{best08}.

On the other hand, some graphs are known not to admit uniform mixing: 
\begin{enumerate}[(a)]
	\item 
	$K_n$ for $n > 4$.
	\item
	$H(n,q)$ for $q > 4$.
	\item
	Cycles $C_n$ for $n = 2^uq$, where $q = 1$ and $u \geq 3$ or $q \cong 3 \mod 4$ \cite{adamczak07}.
	\item
	The cycle $C_5$ \cite{carlson07}.
	\item
	Complete multipartite graphs (except $C_4$) \cite{ahmadi03}.
	\item
	The transpositions Cayley graph $X(S_n,\{\text{transpositions}\})$ 
	\cite{gerhardt03}.
\end{enumerate}

An obvious question raised by these results is which cycles admit uniform mixing.
In this work we show that the cycle $C_n$ does not admit uniform mixing if
$n$ is even and greater than four, or if $n$ is a prime greater than three.
We see that if a bipartite graph on $n$ vertices admits uniform mixing then
$n$ must be divisible by four; if in addition the graph is regular then
$n$ must be the sum of two integer squares.

The graphs we have that do admit uniform mixing are highly regular, and this
leads us to ask which strongly regular graphs admit uniform mixing. We use work of Chan \cite{chan08} to show that the only strongly regular graphs
that admit instantaneous uniform mixing are the Paley graph of order 9 and certain graphs
corresponding to regular symmetric Hadamard matrices with constant diagonal.

Since uniform mixing seems to be rare, we consider a relaxed version, called $\epsilon$-uniform mixing, and we demonstrate that this does take place on all cycles of prime length.

\section{Uniform mixing}
Let $A$ denote the adjacency matrix of a graph $X$ on $n$ vertices. The transition 
operator $U(t)$ defined by
\[
	U(t) = e^{itA}.
\]
determines a continuous quantum walk. It is straightforward to observe that the transition matrix satisfies the following properties.

\begin{enumerate}[(i)]
\item{$U(t)^{T} = U(t).$}
\item{$\overline{U(t)} = U(-t).$}
\item{$U(t)$ is unitary.}
\end{enumerate}

Suppose that $A$ and $B$ are two $m \times n$ matrices. The \textsl{Schur product} $A \circ B$ is the $m \times n$ matrix given by
\[
(A \circ B)_{j,k} = A_{j,k} B_{j,k}.
\]
We say the graph $X$ admits \textsl{uniform mixing} at time $t$ if and only if
\[
U(t) \circ U(-t) = \frac{1}{n}J
\]
where $J$ is the $n \times n$ all-ones matrix.

Let $|| \cdot ||$ denote the Frobenius norm. More explicitly, for two $n \times n$ matrices $A$ and $B$ we have
\[
|| A - B || = \sqrt{\sum_{j = 1}^{n}\sum_{j=1}^{n} \left| A_{j,k} - B_{j,k} \right|^{2}}.
\]
We say that a graph $X$ admits $\epsilon$-uniform mixing if and only if for every $\epsilon > 0$, there exists some time $t$ such that the corresponding mixing matrix $U(t)$ satisfies
\[
|| U(t) \circ U(t)^{*} - \frac{1}{n}J|| < \epsilon. 
\]

It is useful to note the following lemma, which relates uniform mixing on a graph and its complement at certain times.

\begin{lemma}
	Let $X$ denote a regular graph on $n$ vertices, and let $t$ denote an integer 
	multiple of $2 \pi/n$. At time $t$, uniform mixing occurs on $X$ if and only 
	if it occurs on the complement $\overline{X}$.
\end{lemma}

\proof
Note that
\begin{multline*}
	e^{it\bar{A}} = e^{it(J-I-A)} = e^{it(J-I)}e^{-itA} \\ = 		
		\left(e^{i(n-1)t}\left(\frac{1}{n}J\right) + 	
		e^{-it}\left(I-\frac{1}{n}J\right)\right)e^{-iAt}.
\end{multline*}
If $t$ denote an integer multiple of $2 \pi/n$, then the above equation reduces to
\[
	e^{it\bar{A}} = e^{-it}e^{itA},
\]
which implies that the complement of $X$ has uniform mixing if and only if $X$ does.\qed

In Section~\ref{srg_section}, we see that the previous lemma applies to all cases of 
strongly regular graphs that admit uniform mixing.

\section{Type-II Matrices}
Suppose that $A$ is an $n \times m$ matrix with all non-zero entries. We say that the \textsl{Schur inverse} of $A$ is the matrix given by
\[
A^{(-)}_{j,k} =  \frac{1}{A_{j,k}}.
\]
We say that a complex $n \times n$ matrix $H$ is a \textsl{type-II matrix} if and only if
\[
H H^{(-)T} = nI.
\]
\begin{lemma}
The graph $X$ admits uniform mixing at time $t$ if and only if $\sqrt{n}U(t)$ is a type-II matrix.
\end{lemma}
\begin{proof}
We note that a complex number $x$ satisfies $|x| = 1$ if and only if $\bar{x} = x^{-1}$. Therefore the transition matrix satisfies 
\[
U(t) \circ U(t)^{*} = \frac{1}{n}J \quad \text{if and only if} \quad \sqrt{n}U(t)^{*} = \frac{1}{\sqrt{n}} U(t)^{(-)T}. 
\]
Since $U(t)$ is a symmetric unitary matrix, it follows that $U(t)^{(-)T} = U(t)^{-1}$. We conclude that $U(t)$ is flat if and only if $\sqrt{n} U(t)$ is type-II. \qed
\end{proof}

More specifically, a flat type-II matrix is called a \textsl{complex Hadamard matrix}. Therefore, we have the following.

\begin{corollary}
The graph $X$ admits uniform mixing at time $t$ if and only if $\sqrt{n}U(t)$ is a complex Hadamard matrix.\qed
\end{corollary}

For the known examples of graphs that admit uniform mixing, the entries of $\sqrt{n}U(t)$ at the corresponding time are roots of unity. Such matrices are called \textsl{Butson-type} complex Hadamard matrices. They live in a more general family of complex Hadamard matrices whose entries are all algebraic numbers. 

We recall a useful formulation of the Gelfond-Schneider Theorem due to Michel Waldschmidt \cite{burger04}.

\begin{theorem}
\label{gelfond_schneider}
If $x$ and $y$ are two non-zero complex numbers with $x$ irrational, then at least one of the numbers $x$, $e^{y}$, or $e^{xy}$ is transcendental.\qed
\end{theorem}

The above theorem is crucial in the proof of the following result.

\begin{theorem}
\label{no_algebraic}
Let $X$ denote a graph. If at some non-zero time $t$ all entries of the transition matrix $U(t)$ are algebraic numbers, then the ratio of any two non-zero eigenvalues of $X$ must be rational.
\end{theorem}

\begin{proof}
Let $A$ denote the adjacency matrix of $X$, and let $\{\theta_0, \dots, \theta_d\}$ denote
the eigenvalues of $A$. We consider a fixed pair of eigenvalues $\theta_r$ and $\theta_s$.
Since $A$ is an integer matrix, the characteristic polynomial of $A$ is a monic polynomial with integer coefficients. Therefore each eigenvalue of $A$ is an algebraic integer. Furthermore, since $A$ is symmetric, we decompose $A$ as
\[
	A = \sum_{k = 0}^{d} \theta_k E_k,
\]
where each $E_k$ is orthogonal projection onto the $k$-th eigenspace. This directly 
implies that
\begin{equation}
	U(t) = \sum_{k=0}^{d} e^{\theta_k it} E_k.
\end{equation}
From this we see that the eigenvalues of $U(t)$ are $\{e^{it \theta_0}, \dots, e^{it
\theta_d}\}$. If all entries of $U(t)$ are algebraic, then the eigenvalues of $U(t)$ must be
algebraic. If we further suppose that $\theta_r / \theta_s$ is irrational, then Theorem
\ref{gelfond_schneider} implies that one of
\[
	\{ \theta_r / \theta_s, e^{it \theta_s}, e^{i t \theta_s (\theta_r/\theta_s)} \}
\]
must be transcendental, which is a contradiction. Therefore $\theta_r / \theta_s$ 
must be rational for every pair of distinct eigenvalues $\theta_r$ and $\theta_s$ 
of $X$.\qed
\end{proof}

Restricting our consideration to regular graphs yields the following corollary.

\begin{corollary}
\label{no_algebraic_reg}
	Suppose that $X$ is a regular graph with transition matrix $U(t)$. Further suppose that $H$ is a matrix with all algebraic entries such that
	\[
	U(t) = \gamma H,
	\]
	for some non-zero time $t$ and some $\gamma$ in $\cx$. Then $X$ must have all integral eigenvalues.
\end{corollary}

\proof
By a well-known result from linear algebra, we know that for the complex matrix $itA$ we have
\[
\det(U(t)) = \det(e^{itA}) = e^{\tr(itA)} = 1.
\]
Under our assumption that $U(t) = \gamma H$, this result implies that
\[
\gamma^n \det(H) = 1.
\]
Since $H$ has all algebraic entries, it follows that $\det(H)$ is an algebraic number. This implies that $\gamma$ is algebraic, and hence all of the entries of $U(t)$ must be algebraic. Theorem \ref{no_algebraic} shows that the ratio of any two eigenvalues must be rational. Since $X$ is regular, the largest eigenvalue is an integer, and so the rest of the eigenvalues of $X$ must be rational. Recall that the eigenvalues of $A$ are algebraic integers. The only rational algebraic integers are the integers themselves, and so we conclude that all of the eigenvalues of $X$ must be integers.\qed

\section{Complex Hadamards and SRGs}
\label{srg_section}
Let $X$ be a strongly regular graph with parameters $(n,k,\la,\mu)$ and let $A$ be 
its adjacency matrix, with spectral decomposition
\[
A = kE_k + \theta E_\theta + \tau E_\tau,
\]
where $k > \lambda \geq \tau$.

The following result of Chan \cite{chan08} is instrumental in our classification of strongly regular graphs that admit uniform mixing.

\begin{theorem}
	\label{thm:chan}
	Let $A$ be the adjacency matrix of a primitive strongly regular graph $X$ 
	with eigenvalues $k$, $\th$, $\tau$, and $W = I + xA +y\bar{A}$. If $W$ is 
	a complex Hadamard matrix, then $X$ or $\bar{X}$ has one of the following 
	parameter sets $(n,k,\la,\mu)$:
	\begin{itemize}
		\item[(i)] $(4\th^2,2\th^2- \th, \th^2-\th,\th^2-\th)$ 
		\item[(ii)] $(4\th^2,2\th^2+ \th, \th^2+\th,\th^2+\th)$ 
		\item[(iii)] $(4\th^2-1,2\th^2, \th^2,\th^2)$ 
		\item[(iv)] $(4\th^2+4\th+1,2\th^2+2\th, \th^2+\th-1,\th^2+\th)$ 
		\item[(v)] $(4\th^2+4\th+2,2\th^2+ \th, \th^2-1,\th^2)$ 
	\end{itemize}
\end{theorem}

Moreover, Chan determines the possible values of $x$ and $y$ that can occur in each of these cases.

\section{Uniform mixing on SRGs}

If we suppose a strongly regular graph $X$ admits uniform mixing, then there are 
numbers $c$, $x$ and $y$ with $|c| = |x| = |y| = 1$ such that
\[
	c\sqrt{n}U(t) = I + xA + y\bar{A},
\]
and $c \sqrt{n} U(t)$ is a complex Hadamard matrix. Applying Chan's results in the previous
section, we obtain the following classification of primitive strongly regular graphs that
admit uniform mixing. Our main result in this section follows.

\begin{theorem}
	\label{srg_theorem}
	A primitive strongly regular graph $X$ with adjacency matrix $A$ has uniform 
	mixing if and only if one of the following holds
	\begin{itemize}
		\item [(a)] $J-2A$ is a regular symmetric Hadamard matrix of order $4\theta^2$ 
		with constant diagonal and positive row sum and $\theta$ is even.
		\item [(b)] $J- 2A-2I$ is a regular symmetric Hadamard matrix of order $4\theta^2$ 
		with constant diagonal and positive row sum and $\theta$ is odd.
		\item [(c)] The Paley graph of order 9, which has parameters $(9,4,1,2)$.
	\end{itemize}
\end{theorem}

To prove this theorem, we consider each possible case of Theorem \ref{thm:chan} separately. First we lay some groundwork that applies to all cases. Recall that we have
\[
	e^{iAt} = e^{ikt}E_k + e^{i\th t} E_\th + e^{i\tau t} E_\tau,
\]
 and so we can expand both sides of this equation in terms of the spectral decomposition. Comparing the expressions for $E_k$, $E_\th$ and $E_\tau$ on both sides, we get the following system of equations:
\begin{align}
ce^{ikt} \sqrt{n} & = 1 + xk + y(v-k-1) \label{eqn:flatsrg1} \\
ce^{i\th t} \sqrt{n} & = 1 + x\th + y(-\th-1) \label{eqn:flatsrg2} \\
ce^{i\tau t} \sqrt{n} & = 1 + x\tau + y(-\tau-1)\label{eqn:flatsrg3}.
\end{align}
These equations are the \textsl{characterizing equations}. The graph $X$ admits uniform mixing if and only if a solution to the characterizing equations exists in $t$, $x$, $y$ and $c$ with $|x| = |y| = |c| = 1$.

Chan shows \cite[Lemma~2.2]{chan08} that if $X$ has at least $5$ vertices,  then $\th + \tau \in \{0,1,2\}$. Moreover, for each value of $\th+\tau$, Chan finds parameter sets $(n,k,\la,\mu)$ in Theorem \ref{thm:chan} that  can occur and also determines the values of $x$ and $y$ that can occur in the characterizing equations.  

In each case, we must also consider if the complement has uniform mixing $ce^{it\bar{A}} = I + xA + y\bar{A}$. The eigenvalues of the complement are
\[
\{v-k-1,-\th-1,-\tau-1\}.
\]
The characterizing equations for the complement of $X$, written in terms of the parameters of $X$, are:
\begin{align}
ce^{i(n-k-1)t} \sqrt{n} & = 1 + xk + y(n-k-1) \label{eqn:flatsrgc1} \\
ce^{i(-\th-1) t} \sqrt{n} & = 1 + x\th + y(-\th-1) \label{eqn:flatsrgc2} \\
ce^{i(-\tau-1) t} \sqrt{n} & = 1 + x\tau + y(-\tau-1)\label{eqn:flatsrgc3} .
\end{align}

\subsection{Regular symmetric Hadamard matrices}

We first consider uniform mixing on strongly regular graphs with parameters given in (i) or (ii) of Theorem~\ref{thm:chan}. If a graph $X$ with adjacency matrix $A$ has these parameters or the complementary parameters, then $J-2A$ or $J-2A-2I$, respectively, is a regular symmetric Hadamard matrix \cite{brouwer84}.

\begin{lemma} 
    Suppose $X$ is a primitive strongly regular graph such that $X$ or $\overline{X}$ 
    has parameters given in (i) or (ii) of Theorem \ref{thm:chan}. Then $X$ admits uniform 
    mixing if and only if one of the following holds:
    \begin{enumerate}
        \item [(a)] 
		The parameters of $X$ or $\overline{X}$ are given in (i) and $\theta$ is even. 
        \item [(b)] 
		The parameters of $X$ or $\overline{X}$ are given in (ii) and $\theta$ is odd.
    \end{enumerate}
\end{lemma}

\proof
First suppose that $X$ is a graph with parameters $(4\th^2,2\th^2- \th, \th^2-\th,\th^2-\th)$. Chan shows that $x = -1$ and $y = 1$. The characterizing equations above reduce to 
\begin{align*}
ce^{i\th(2\th-1)t}  & = 1  \\
ce^{i\th t} & = -1 \\
ce^{-i\th t}  & = 1.
\end{align*}
Thus $c=e^{i\th t} = \pm i$, so $\th t = \pi(m+1/2)$ for some integer $m$, and $t = \pi(m+1/2)/\th$ and $kt = (2\th-1)\pi(m+1/2)$. When $\th$ is even, any $m$ is a valid solution and uniform mixing occurs. These are the Latin square graphs $L_\th(2\th)$. When $\th$ is odd, no solutions occur. 

For the complement, the characterizing equations for $\overline{G}$ reduce to 
\begin{align}
ce^{i(2\th^2+\th-1)t} & = 1  \\
ce^{i(\th-1) t} & = 1  \\
ce^{i(-\th-1)t} & = -1
\end{align}
for some $\abs{c} = 1$.
Thus $c = e^{-i(\th-1) t}$ and the first and third equations are $e^{it2\th^2t} = 1$ and $e^{-i2\th t} = -1$. Thus $e^{i\th t} = \pm i$, so $\th t = \pi(m+1/2)$ for some $m$, and $t = \pi(m+1/2)/\th$. When $\th$ is even, any $m$ is a valid solution and uniform mixing occurs. When $\th$ is odd, no solutions occur. 

If $X$ or $\overline{X}$ has parameters $(4\th^2,2\th^2+ \th, \th^2+\th,\th^2+\th)$, then a similar results holds: there are no solutions for $\th$ even, and for $\th$ odd $t = \pi(m+1/2)/\th$ is a solution for every $m$. These are the negative Latin square graphs $NL_\th(2\th)$. The same applies to the complement: uniform mixing occurs when $\th$ is even, and no solutions occur when $\th$ is odd.\qed

There are infinite families of strongly regular graphs with these parameters \cite{haemers08}.

\subsection{Symplectic-type graphs}   
                             
We turn the graphs in part (iii) of Theorem~\ref{thm:chan}.
These are the strongly regular graphs with parameters $(4\th^2-1,2\th^2, \th^2,\th^2)$; their eigenvalues satisfy $\th + \tau = 0$. Chan shows in this case that $(x,y)$ is one of
\[
	\left\{\left(-1, \frac{2\th^2-3 \pm i\sqrt{4\th^2-5}}{2(\th^2-1)}\right),
		\left(\frac{-2\th^2+1 \pm i\sqrt{4\th^2-1}}{2\th^2},1\right)\right\}.
\]

{\bf Case} $(x,y) = (\tfrac{-2\th^2+1 \pm i\sqrt{4\th^2-1}}{2\th^2},1)$: The characterizing equations reduce to 
\begin{align}
ce^{i2\th^2t} \sqrt{n} & = \pm i\sqrt{n} \label{eqn:flatsrg1.1} \\
ce^{i\th t} \sqrt{n} & = \th(x-1) \label{eqn:flatsrg1.2}  \\
ce^{-i\th t} \sqrt{n} & = -\th(x-1) \label{eqn:flatsrg1.3} .
\end{align}
Equations \eqref{eqn:flatsrg1.2} and \eqref{eqn:flatsrg1.3} imply that $e^{2i\th t} = -1$ and so $e^{i\th t} = \pm i$ and $e^{i2\th^2t} = (-1)^\th$. Then Equation \eqref{eqn:flatsrg1.1} implies that $c = \pm i$. But then the LHS of equation \eqref{eqn:flatsrg1.2} is real while the RHS is not, a contradiction. So there are no solutions.

The characterizing equations for the complement of $X$ are:
\begin{align}
ce^{i(2\th^2-2)t} \sqrt{n} & = \pm i\sqrt{n} \label{eqn:flatsrgc1.1} \\
ce^{i(-\th-1)t} \sqrt{n} & = \th(x-1) \label{eqn:flatsrgc1.2}  \\
ce^{i(\th-1)t} \sqrt{n} & = -\th(x-1) \label{eqn:flatsrgc1.3} .
\end{align}
Equations \eqref{eqn:flatsrgc1.2} and \eqref{eqn:flatsrgc1.3} imply that $e^{i2\th t} = -1$ and so $e^{i2\th^2t} = (-1)^\th$. Then equation \eqref{eqn:flatsrgc1.1} implies that $c = \pm i e^{i2t}$. Since $e^{it}$ is a $4\th$-th root of unity, so is $ce^{i(-\th-1)t} = e^{it(1-\th)}$. On the other hand, from equation \eqref{eqn:flatsrgc1.2},
\[
e^{it(1-\th)} = \tfrac{\th}{\sqrt{n}}(x-1) = \tfrac{-\sqrt{n} \pm i}{2\th},
\]
which has algebraic degree exactly four since $n$ is not a perfect square. This implies that $e^{it(1-\th)}$ is a primitive $m$-th root of unity for some $m$ such that $\phi(m) = 4$, namely $m \in \{5,8,10,12\}$. (Here $\phi$ is the Euler totient function.) But then it is not difficult to check that none of these primitive $m$-th roots of unity have imaginary part $\pm1/2\th$, for any $\th > 1$. Thus there are no solutions.

\paragraph{}

{{\bf Case}  $(x,y) = (-1, \tfrac{2\th^2-3 \pm i\sqrt{4\th^2-5}}{2(\th^2-1)})$:} The characterizing equations reduce to 
\begin{align}
ce^{i2\th^2t} \sqrt{n} & = -2 \pm i\sqrt{4\th^2-5}  \label{eqn:flatsrg2.1} \\
ce^{i\th t} \sqrt{n} & = 1 - \th + y(-\th-1) \label{eqn:flatsrg2.2} \\
ce^{-i\th t} \sqrt{n} & = 1 + \th + y(\th-1) \label{eqn:flatsrg2.3}.
\end{align}
Simplifying \eqref{eqn:flatsrg2.2} plus \eqref{eqn:flatsrg2.3} we have $\sqrt{n} c\Re(e^{i\th t}) = 1 - y$.  Taking the absolute value, we get $\Re(e^{i\th t})^2 = 1/n(\th^2-1)$. Similarly, from \eqref{eqn:flatsrg2.2} minus \eqref{eqn:flatsrg2.3} we get $\Im(e^{i\th t})^2 = \th^2(4\th^2 - 5)/n(\th^2-1)$. Thus we can solve:
\begin{align}
e^{i\th t} & = \frac{1}{\sqrt{n(\th^2-1)}}[\pm 1 \pm \th i\sqrt{4\th^2-5}] \label{eqn:eitht}, \\
c & = \frac{1}{2\sqrt{\th^2-1}}[\pm 1 \pm i\sqrt{4\th^2-5}] \label{eqn:c}.
\end{align}
We then check that there are no solutions to the characterizing equations for any choice of integer $\th$. For small $\th$ we can check explicitly. When $\th$ is large, we examine $\arg(ce^{i2\th^2 t})$ in two different ways: once from combining \eqref{eqn:eitht} and \eqref{eqn:c} and once from \eqref{eqn:flatsrg2.1}. We use asymptotics with error bounds to show the values are not equal. We need one preliminary observation: when $\th$ is large, $e^{i \th t}$ is close to $\pm i$ and so the argument of $e^{i\th t}$ is approximately $\pm (\pi/2 - \Re(e^{i\th t}))$. 

\begin{lemma}\label{lem:argbound}
For any $x$ on the unit circle, $\arg(x) = \pm (\pi/2 - \Re(x)) + \eps$, where $\abs{\eps} \leq \Re(x)^2$.
\end{lemma}

\proof
Assume $\Im(x) > 0$ and $\Re(x) < 0$: the other cases are similar. The part of $\arg(x)$ past $\pi/2$ is the portion of the unit circle arc in the upper left quadrant, which is contained in a rectangle with horizontal width $\Re(x)$ and vertical height $1 - \sqrt{1-\Re(x)^2}$. Since $\sqrt{z} \geq z$ for any $z \in [0,1]$, we have
\begin{align*} \arg(x) - \pi/2 & \leq -\Re(x) + 1 - \sqrt{1-\Re(x)^2} \\
& \leq -\Re(x) + 1 - (1-\Re(x)^2) \\
& = -\Re(x) + \Re(x)^2.
\end{align*}
On the the other hand $\arg(x) - \pi/2 \geq -\Re(x)$.\qed

Now consider the arguments of $e^{i\th t}$ and $c$ from equations  \eqref{eqn:eitht} and \eqref{eqn:c}:
\begin{align*}
\th t = \arg(e^{i \th t}) &= \pm \tfrac{\pi}{2} \pm \tfrac{1}{2\th^2} + O(\tfrac{1}{\th^4}), \\ \vspace{4pt}
\arg(c) &= \pm\tfrac{\pi}{2} \pm \tfrac{1}{2\th} + O(\tfrac{1}{\th^2}), \\ \vspace{4pt}
\arg(ce^{i2\th^2 t}) &= \arg(c) + 2\th(\th t) \\ \vspace{4pt}
& = \pm \tfrac{\pi}{2}  \pm \tfrac{1}{2\th}[2 \pm 1] +O(\tfrac{1}{\th^2}). \vspace{4pt}
\end{align*}
On the other hand from  \eqref{eqn:flatsrg2.1} we have $\Re(ce^{i2\th^2 t}) = -2/\sqrt{n} \approx  -\tfrac{1}{\th}$, and so
\[
\arg(ce^{i2\th^2 t}) = \pm \tfrac{\pi}{2}  - \tfrac{1}{\th} +O(\tfrac{1}{\th^2}),
\]
a contradiction. Using Lemma \ref{lem:argbound} this asymptotic argument can be made exact. Details are left to the reader.

The characterizing equations of the complement reduce to 
\begin{align}
ce^{i(2\th^2-2)t} \sqrt{n} & = -2 \pm i\sqrt{4\th^2-5} \label{eqn:flatsrg2.1c} \\
ce^{i(-\th-1) t} \sqrt{n} & = 1 - \th + y(-\th-1) \label{eqn:flatsrg2.2c} \\
ce^{i(\th-1) t} \sqrt{n} & = 1 + \th + y(\th-1) \label{eqn:flatsrg2.3c}.
\end{align}
Adding/subtracting the last two equations gives $\sqrt{n}ce^{-it}\Re(e^{i\th t}) = 1-y$ and $\sqrt{n}ce^{-it}\Im(e^{i\th t}) = \th(1+y)$. Again we can solve:
\begin{align}
e^{i\th t} & = \frac{1}{\sqrt{n(\th^2-1)}}[\pm 1 \pm \th i\sqrt{4\th^2-5}] \\
ce^{-it} & = \frac{1}{2\sqrt{\th^2-1}}[\pm 1 \pm i\sqrt{4\th^2-5}].
\end{align}
Again we use an asymptotic description of the argument which can be made precise with Lemma \ref{lem:argbound}. We get: 
\begin{align*}
\th t = \arg(e^{i \th t}) &= \pm \tfrac{\pi}{2} \pm \tfrac{1}{2\th^2} + O(\tfrac{1}{\th^4}), \\
t &= \pm  \tfrac{\pi}{2\th} \pm \tfrac{1}{2\th^3} + O(\tfrac{1}{\th^5}),\\
\arg(c) & = t  \pm  \tfrac{\pi}{2} \pm \tfrac{1}{2\th} + O(\tfrac{1}{\th^2}), \\ 
\arg(ce^{i(2\th^2-2)t}) &= \arg(c) + (2\th^2-2)t \\
& = \pm  \tfrac{\pi}{2} + \tfrac{1}{\th}[\pm \tfrac{1}{2} \pm 1 \pm \pi/2] + O(\tfrac{1}{\th^2}).
\end{align*}
Meanwhile from equation \eqref{eqn:flatsrg2.1c}, 
\[
\arg(ce^{i(2\th^2-2)t}) = \pm  \tfrac{\pi}{2} - \tfrac{1}{\th} + O(\tfrac{1}{\th^2}).
\] 
Since the two expressions are not equal, no solution exists. Thus there is no uniform mixing in the symplectic-type graphs or their complements.

\subsection{Regular conference matrix type graphs}  

We treat the fourth family of graphs from Chan's theorem.
These have parameters $(n,k,\la,\mu) = (4\th^2+4\th+2,2\th^2+ \th, \th^2-1,\th^2)$ 
and $\th + \tau = -1$.

Chan found $x = \pm i$ or $x = \tfrac{-1 \pm i\sqrt{(4\th^2(\th+1)^2-1}}{2\th(\th+1)}$ and in either case, the $y = \bar{x}$. The characteristic equations reduce to:
\begin{align}
ce^{ikt} \sqrt{n} & = 1 + xk + y(n-k-1) \label{eqn:flatsrg4.1} \\
ce^{i\th t} \sqrt{n} & = 1 + x\th + \bar{x}(-\th-1) \label{eqn:flatsrg4.2} \\
ce^{i(-\th-1) t} \sqrt{n} & = 1 + x(-\th-1) + \bar{x}\th \label{eqn:flatsrg4.3} .
\end{align}
From \eqref{eqn:flatsrg4.2} plus \eqref{eqn:flatsrg4.3} we have 
\[
\sqrt{n}c(e^{i\th t} + e^{i(-\th-1)t}) = 2 - 2\Re(x),
\]
so $ce^{i\th t} + ce^{i(-\th-1)t}$ is real. Thus either $ce^{i\th t} = -ce^{i (-\th-1)t}$ or $ce^{i\th t} = \overline{ce^{i (-\th-1)t}}$. But in the first case $x = y = 1$, which is not a solution. We conclude that  or $ce^{i\th t} = \overline{ce^{i (-\th-1)t}}= \bar{c}e^{i(\th+1)t}$ and therefore $c = e^{it/2}$. The same analysis shows that $c = e^{it/2}$ for the complement of $G$ as well.

\paragraph{}

{{\bf Case} $x = i$, $y = -i$:} The characteristic equations reduce to:
\begin{align} 
\sqrt{n}e^{it(k+1/2)} &= 1 -i(2\th+1) \\
\sqrt{n}e^{it(\th+1/2)} &= 1 +i(2\th+1) \label{eqn:algdeg4} 
\end{align}
It follows that $e^{it(k+\th+1)} = 1$, so $e^{it}$ is a $(k+\th+1) = (2\th^2+2\th+1)$-th root of unity and $e^{it(\th+1/2)}$ is a $(4\th^2+4\th+2)$-th root of unity. But by \eqref{eqn:algdeg4}, noting that $\sqrt{n}$ is not an integer, $e^{it(\th+1/2)}$ has algebraic degree exactly $4$. Since none of the primitive $m$-th roots of unity of algebraic degree $4$ have real part $1/\sqrt{n}$ for any $\th$, there are no solutions.

For the complement of $X$, the characteristic equations reduce to
\begin{align} 
\sqrt{n}e^{it((n-k-1)+1/2)} &= 1 -i(2\th+1) \\
\sqrt{n}e^{it(\th+1/2)} &= 1 -i(2\th+1) \label{eqn:algdeg4.2} 
\end{align}
So $e^{it(n-k-1-\th)} = 1$ and $e^{it}$ is a $n-k-1-\th = (2\th^2+2\th+1)$-th root of unity. Again there are no solutions.

\paragraph{}

{{\bf Case} $x = \tfrac{-1 \pm i\sqrt{(4\th^2(\th+1)^2-1}}{2\th(\th+1)}$:} This case is similar. The characteristic equations reduce to:
\begin{align*} 
\sqrt{n}e^{it(k+1/2)} &= \frac{-2\th^2-2\th-1 \pm i(2\th+1)\sqrt{4\th^2(\th+1)^2-1}}{2\th(\th+1)} \\
\sqrt{n}e^{it(\th +1/2)} &= \frac{2\th^2+2\th+1 \mp i(2\th+1)\sqrt{4\th^2(\th+1)^2-1}}{2\th(\th+1)} 
\end{align*}
Thus $-e^{it(\th+1/2)} = e^{it(k+1/2)}$ and so $e^{it(k-\th)} = -1$ and $e^{it}$ is a $2(k-\th) = 4\th^2$-th root of unity while $e^{it(k+1/2)}$ is an $8\th^2$-th root of unity. Since none of the primitive $m$-th roots of unity of algebraic degree $4$ have real part $\tfrac{2\th^2+2\th+1}{\sqrt{n}(2\th(\th+1)}$ for any $\th$, there are no solutions.

For the complement of $X$, the characteristic equations are:
\begin{align*} 
\sqrt{n}e^{it((n-k-1)+1/2)} &= \frac{-2\th^2-2\th-1 \pm i(2\th+1)\sqrt{4\th^2(\th+1)^2-1}}{2\th(\th+1)} \\
\sqrt{n}e^{it(\th +1/2)} &= \frac{2\th^2+2\th+1 \pm i(2\th+1)\sqrt{4\th^2(\th+1)^2-1}}{2\th(\th+1)} 
\end{align*}
Then $e^{it(n-k+\th)} = 1$, so $e^{it}$ is a $(n-k+\th) = (2\th^2+4\th+2)$-th root of unity. Again no solutions exist.

\subsection{Conference graphs}
                              
The final family of graphs that arise in Chan's theorem are known as conference graphs.
A conference graph has parameters $((2\th+1)^2,2\th^2+2\th, \th^2+\th-1,\th^2+\th)$ with $\th
+ \tau = -1$. Chan shows that $y = \bar{x}$. As in the previous case, the characteristic
equations reduce to \eqref{eqn:flatsrg4.1}, \eqref{eqn:flatsrg4.2}, and
\eqref{eqn:flatsrg4.3} with $c = e^{it/2}$. Then equations \eqref{eqn:flatsrg4.2} minus
\eqref{eqn:flatsrg4.3} reduce to $\Im(x) = \Im(e^{it(\th+1/2)})$ and therefore $x =
e^{it(\th+1/2)}$ or $x = -e^{-it(\th+1/2)}$. There are two possible values of $(x,y)$, namely
\[
	x \in \left\{\frac{-1 \pm \sqrt{(2\th+1)(2\th-1)}i}{2\th},  
		\frac{1 \pm \sqrt{(2\th+1)(2\th+3)}i}{2(\th+1)}\right\}, \; y = \bar{x}.
\]

{{\bf Case} $x = \tfrac{-1 \pm \sqrt{(2\th+1)(2\th-1)}i}{2\th}$:} In this case
\eqref{eqn:flatsrg4.1} reduces to $ce^{ikt} = -1$. Thus $t = \frac{\pi(2m+1)}{k+1/2} =
\tfrac{2\pi(2m+1)}{n}$, for some integer $m$.

Suppose $\th$ is an integer. Now $x = e^{it(\th+1/2)} = e^{i\pi(2m+1)/(2\th+1)}$ or
$-e^{-it(\th+1/2)} = -e^{-i\pi(2m+1)/(2\th+1)}$, so $x$ is a $2(2\th+1)$-th root of unity.
But we claim that $x$ is a $2(2\th+1)$-th root of unity only when $\th = 1$. To see this,
note that $\abs{\Re(x)} = 1/2\th$, while any $2(2\th+1)$-th root of unity has a real part of
larger absolute value. (The root smallest real part in absolute value is $\om =
e^{i\pi(\th+1)/(2\th+1)}$, which has $\abs{\Re(\om)} = \sin(\tfrac{\pi}{2(2\th+1)}) > 1/2\th$
for $\th > 1$.) In the case $\th = 1$, we get the $3 \times 3$ Latin square graph, for which
uniform mixing does occur.

If $\th$ is not an integer, then by Corollary \ref{no_algebraic_reg} uniform mixing cannot occur on the corresponding conference graph. So there are no solutions if $\th$ is not an integer.

{{\bf Case} $x = \tfrac{1 \pm \sqrt{(2\th+1)(2\th+3)}i}{2(\th+1)}$:} This case is similar: we
find $ce^{ikt} = 1$. Thus $t = \frac{\pi(2m+1)}{k+1/2} = \tfrac{2\pi(2m+1)}{n}$, for some
integer $m$. So $t(\th+1/2) = (2\pi m)/(2\th+1)$, and $x = e^{it(\th+1/2)}$ or
$-e^{-it(\th+1/2)}$ is a $(2\th+1)$-th root of unity. Since $\abs{\Re(x)}$ is too small, this
never occurs. Thus the only conference graph having uniform mixing is the $3 \times 3$ Latin
square.

The combination of the work above proves Theorem \ref{srg_theorem}.

\section{Bipartite graphs}
\label{bipartite_graphs} 
Next we consider bipartite graphs. To begin, suppose that $X$ is a bipartite graph on $n$ vertices with adjacency matrix $A$. We assume the rows and columns of $A$ are ordered such that 
\[
	A =\begin{pmatrix} 0&B\\ B^T&0 \end{pmatrix}.
\]
Since $A$ is a block matrix, we easily compute the even and odd powers of $A$.
\begin{align*}
	A^{2k} &= \pmat{(BB^T)^k&0\\ 0&(B^TB)^k} \\
     A^{2k+1} &= \pmat{0&(BB^T)^kB\\ (B^TB)^kB^T&0}.
\end{align*}
Now we derive an expression for the transition matrix of $X$ in terms of these blocks.
\begin{align*}
	U(t) &= e^{itA} = \sum_{k \geq 0} \frac{1}{k!}(itA)^k \\
	       &= \sum_{k \geq 0}\frac{(-1)^kt^{2k}}{(2k)!} \pmat{(BB^T)^{k}&0\\ 0&(B^TB)^{k}} \\
	       &+ i \sum_{k \geq 0}\frac{(-1)^{k}t^{2k+1}}{(2k+1)!} \pmat{0&(BB^T)^{k}B\\ (B^TB)^{k}B^T&0} \\
\end{align*}

To simplify notation we write $U(t)$ in block form as
\[
	U(t) = \pmat{F_1(t)& iK(t)\\ iK^T(t)&F_2(t)}.
\]
We pay particular attention to the fact that $F_1(t)$, $F_2(t)$, and $K(t)$ are all real matrices for all times $t$. This observation leads to the following result.

\begin{lemma}
\label{bipartite_fourth_root}
Suppose $X$ is a bipartite graph on $n > 2$ vertices with transition matrix $U(t)$. If $X$ admits uniform mixing at time $t$, then each entry of $\sqrt{n}U(t)$ is a fourth root of unity.
\end{lemma}
\begin{proof}
If $X$ admits uniform mixing at time $t$, then 
\[
\left| U(t)_{j,k} \right| = \frac{1}{\sqrt{n}}
\]
for all $1 \leq j,k \leq n$. In terms of the blocks that comprise $U(t)$, this means that all of the entries of the real matrices $F_1(t)$, $F_2(t)$ and $K(t)$ are equal to $\pm 1 / \sqrt{n}$. Hence all entries of $\sqrt{n}U(t)$ are fourth roots of unity.\qed
\end{proof}

Lemma \ref{bipartite_fourth_root} can be used to rederive Kay's result \cite{Kay11} concerning the phase factors for perfect state transfer in bipartite graphs. In addition, combining Lemma \ref{bipartite_fourth_root} with Theorem \ref{no_algebraic} yields the following results.

\begin{lemma}
If $X$ is a bipartite graph on $n > 2$ vertices that admits uniform mixing, then the ratio of any two non-zero eigenvalues must be rational.
\end{lemma}
\begin{proof}
Let $U(t)$ denote the transition matrix of $X$, and suppose that $X$ admits uniform mixing at non-zero time $t$. By Lemma \ref{bipartite_fourth_root}, we know that each entry of $U(t)$ must be a fourth root of unity. Therefore, all the entries of $U(t)$ are algebraic numbers. By Theorem \ref{no_algebraic}, the ration of any two non-zero eigenvalues of $X$ must be rational.\qed
\end{proof}

\begin{theorem}
\label{bipartite_divisibility_theorem}
 If $X$ is a bipartite graph on $n > 2$ vertices that admits uniform mixing, then $n$ is divisible by four.
\end{theorem}
\begin{proof}
Let $U(t)$ denote the transition matrix of $X$, and suppose that $X$ admits uniform mixing at time $t$. By Lemma \ref{bipartite_fourth_root}, we know that each entry of $U(t)$ must be a fourth root of unity.

Let $\Gamma_1$ and $\Gamma_2$ denote the two colours classes on $X$. For convenience, we assume that the vertices in $\Gamma_1$ correspond to the first $|\Gamma_1|$ rows of $U(t)$. Let $D$ denote the diagonal matrix of order $n\times n$ such that 
\[
D_{u,u} = \begin{cases} 1 \; \; \text{if}  \; \; u \in \Gamma_1 \\ i \; \; \text{if} \; \; u \in \Gamma_2. \end{cases} 
\]
Let $H$ denote the matrix given by
\[
	H = \sqrt{n} DU(t)D.
\]
In terms of the blocks of $U(t)$, we see that
\[
	H = \sqrt{n}\pmat{F_1& -K\\ -K^T&-C_2}.
\]
This implies that $H$ is a real matrix with entries equal to $1$ or $-1$. A straightforward computation also reveals that
\[
H H^{*} = nDU(t)D D^{*} U(t)^{*} D^{*} = nI.
\]
Therefore $H$ is a real Hadamard matrix. It is well known that if $H$ is a real Hadamard matrix of order $n$ such that $n > 2$, then $n$ is divisible by four.\qed
\end{proof}

The work done by Adamczak et al.~\cite{adamczak07} implies that if uniform mixing occurs on $C_{2m}$, then $m$ must be a sum of two squares. Here we show that this is true for all regular bipartite graphs.

\begin{theorem}
If $X$ is a regular, bipartite graph with $n$ vertices that admits uniform mixing, then $n$ is the sum of two integer squares.
\end{theorem}

\begin{proof}
Suppose $X$ is a regular bipartite graph with transition matrix $U(t)$. Further suppose that $X$ admits uniform mixing at time $t$. By Lemma \ref{bipartite_fourth_root}, each of the entries of $\sqrt{n}U(t)$ is a fourth root of unity. In particular, this implies that
\[
	\sqrt{n} U(t)\one =(a+ib)\one \quad \text{and} \quad \sqrt{n} U(t)^{*}\one =(a-ib)\one 
\]
for some integers $a$ and $b$ where $\one$ is the all-ones vector.  Taking the product of both of these expressions yields
\[
	nU(t)U(t)^{*}\one = (a-ib)(a+ib)\one = (a^2 + b^2) \one.
\]
Since $U(t)$ is unitary, we know that $U(t)U(t)^{*} = I$. We conclude that $n = a^2 + b^2$.\qed
\end{proof}

\section{Cycles}
It was conjectured by Ahmadi et al.~\cite{ahmadi03} that no cycle $C_n$, except for $C_3$ and $C_4$, admits uniform mixing. Adamczak et al.~\cite{adamczak07} show that $C_n$ does not admit uniform mixing if $n = 2^u$ for $u \geq 3$ or if $n = 2^{u}m$ where $m$ is not the sum of two integer squares and $u \geq 1$. Carlson et al.~\cite{carlson07} show that uniform mixing does not occur on $C_5$. (This result also follows from the earlier result in this paper that uniform mixing does not occur on conference graphs.)

Using Corollary~\ref{no_algebraic_reg}, we show that uniform mixing cannot occur on any even cycle. Likewise, we see that uniform mixing cannot occur on $C_p$ where $p$ is a prime such that $p \geq 5$. For both of these results, we use the irrationality of the eigenvalues of the underlying graph.

First we recall that the eigenvalues of a general cycle $C_n$ have the form
\[
\theta_r = \omega^r + \omega^{-r},
\]
for $0 \leq r \leq n-1$ where $\omega = e^{2 \pi i /n}$. If $n$ is even, then $2$ and $-2$ are both simple eigenvalues of $C_n$, and the remaining eigenvalues each have multiplicity 2. On the other hand, if $n$ is odd, then $\theta_0 = 2$ is the only simple eigenvalue, and the nontrivial eigenvalues each have multiplicity two.

\begin{lemma}
\label{irrational_lemma}
    If $n =5$ or $n \geq 7$, then $C_n$ has an irrational eigenvalue. 
\end{lemma}

\proof
Using the notation above, we note that $\theta_1 = 2 \cos(2 \pi i / n)$. Since $n$ is a positive integer, the eigenvalue $\theta_1 = 2 \cos(2 \pi i / n)$ is rational if and only if $n \in \{1,3,4,6\}$ (see \cite{olmsted45}.) If we assume $n = 5$ and $n \geq 7$, then $\theta_1$ is an irrational eigenvalue of $C_n$.\qed

Applying Corollary \ref{no_algebraic_reg}, we obtain the following result.
 
\begin{theorem}
    The cycle $C_4$ is the unique even cycle that admits uniform mixing. 
\end{theorem}

\proof
Let $U(t)$ denote the transition matrix of $C_{2m}$. If $m = 2$, it is straightforward to observe that $U(t)$ is flat at time $t = \pi/4$. Suppose for a contradiction that $m > 2$ and $U(t)$ is flat at some time $t$. By Theorem \ref{bipartite_divisibility_theorem}, it follows that $m$ must be even, and so $m > 3$. Furthermore, by Lemma \ref{bipartite_fourth_root} all of the entries of $U(t)$ are $\pm 1/\sqrt{2m}$ or $\pm i/\sqrt{2m}$, and hence they are all algebraic. Thus by Corollary \ref{no_algebraic_reg}, the eigenvalues of $C_{2m}$ must be integers, which contradicts Lemma \ref{irrational_lemma}.\qed

Next we turn our attention to cycles of odd prime order. It is known that $C_3$ admits uniform mixing at time $t = 2 \pi/9$. On the other hand, it is known that $C_5$ does not admit uniform mixing \cite{carlson07}. (This is also a consequence of our earlier results about conference graphs.) We extend these results to show that $C_p$ does not admit uniform mixing for any prime $p$ such that $p \geq 5$. To obtain this result, we consider the following important observation due to Haagerup \cite{haagerup08}.

\begin{theorem}
    Let $p$ denote a prime number. There are a finite number of cyclic type-II matrices of order $p$ with constant diagonal $1$.\qed
\end{theorem}

Haagerup's work refers to so-called cyclic $p$-roots, which are equivalent to cyclic $p \times p$ type-II matrices with constant diagonal $1$. Using Haagerup's Theorem, we obtain the following new result.

\begin{theorem}
    The cycle $C_3$ is the unique cycle of odd prime order that admits uniform mixing.
\end{theorem}

\proof
Let $p$ denote an odd prime such that $p \geq 5$, and let $U(t)$ denote the transition matrix of $C_p$. Suppose that $C_p$ admits uniform mixing at time $t$. This implies that $\sqrt{p}U(t)$ is a scalar multiple of a type-II matrix of order $p$ with constant diagonal $1$. By Haagerup's work we know that there are a finite such number of such type-II matrices. The set of all cyclic type-II matrices with constant diagonal $1$ of a given order form an algebraic variety. It is well known that if an algebraic variety over $\cx$ has a finite number of points, then the coordinates of each point are algebraic numbers. Therefore $U(t)$ must be a scalar multiple of a matrix with all algebraic entries. In turn, Theorem \ref{no_algebraic_reg} implies that $C_p$ must have all integral eigenvalues. This contradicts the fact that $C_p$ has an irrational eigenvalue, which we proved in Lemma \ref{irrational_lemma}.\qed

\section{Cyclotomic number theory}
We recall some information about number fields related to the cyclic schemes. The following results are well-known. See \cite{samuel70}, for example.

Suppose that $p$ is an odd prime. Let $d = (p-1)/2$, and let $\theta_1, \dots, \theta_d$ denote the nontrivial eigenvalues of $C_p$. Also let $\phi$ denote the Euler phi function. As before, $\omega$ is the primitive $p$-root of unity given by
\[
\omega = e^{2 \pi i/p}.
\]

\begin{lemma}
\label{theta_polynomial_lemma}
Each of the eigenvalues $\theta_0, \dots, \theta_d$ can be expressed as a polynomial in $\theta_1$ with integral coefficients. 
\end{lemma}
\begin{proof}
First recall that $\theta_0 = 2$ is trivially a polynomial in $\theta_1$. Next note that 
\[
(\omega + \omega^{-1})^2 = \omega^2 + \omega^{-2} + 2.
\]
Rearranging this equation yields $\theta_2 = \theta_1^2 -2$. More generally, for $2 \leq k \leq d$, we see that $(\omega + \omega^{-1})^k$ can be expressed as an integral linear combination of $1, \theta_1, \dots, \theta_{k-1}$. By induction we conclude that each eigenvalue $\theta_0, \dots, \theta_d$ can be expressed as a polynomial in $\theta_1$ with integral coefficients.\qed
\end{proof}

Using this lemma, we obtain useful information about the smallest number field containing all of the eigenvalues of the cycle. 
\begin{lemma}
    The extension field $\rats(\theta_1, \dots, \theta_d)$ is isomorphic to $\rats(\theta_1)$, and
    \[
        [\rats(\theta_1) : \rats] = (p-1)/2.
    \] 
\end{lemma}

\begin{proof}
By Lemma \ref{theta_polynomial_lemma}, we immediately see that $\rats(\theta_1, \dots, \theta_d)$ is isomorphic to $\rats(\theta_1)$. Recall that the cyclotomic field $\rats(\omega)$ is an algebraic extension of $\rats$, and
\[
    [\rats(\omega): \rats]  = \phi(p) = p-1,
\]
where $\phi$ denotes the Euler phi function. Since $\omega$ is the root of a quadratic polynomial over $\rats(\theta_1)$ and $\theta_1$ is real, we see that $[\rats(\theta_1) : \rats] = (p-1)/2$.\qed
\end{proof}

Finally, we obtain a very useful theorem about the linear independence of a subset of the eigenvalues of the cycle.

\begin{theorem}
\label{prime_independence}
    The set $\{1, \theta_1, \theta_2, \dots, \theta_{(p-3)/2}\}$ is linearly independent 
    over $\rats$.
\end{theorem}

\begin{proof}
Suppose, for a contradiction, that there exists some rational coefficients $\alpha_k$ for $0 \leq k \leq (p-3)/2$ such that
\[
\alpha_0 + \alpha_1 \theta_1 + \alpha_2 \theta_2 + \dots + \alpha_{(p-3)/2} \theta_{(p-3)/2} = 0,
\]
with at least one $\alpha_k$ non-zero. Since we have $\theta_k = \omega^k + \omega^{-k}$, we see can re-express this equation in terms of $\omega$ as
\[
\alpha_0 + \alpha_1( \omega + \omega^{-1}) + \alpha_2 (\omega^2 + \omega^{-2}) + \dots + \alpha_{(p-3)/2} (\omega^{(p-3)/2} + \omega^{-(p-3)/2}) = 0.
\]
Multiplying both sides by $\omega^{(p-3)/2}$ yields a polynomial equation in terms of $\omega$ with rational coefficients and degree at most $p-3$. This contradicts the fact that $[\rats(\omega): \rats] = p-1$.\qed
\end{proof}

In our following work with $\epsilon$-uniform mixing on $C_{p}$, we consider the scaled exponents of the eigenvalues of $U(t)$ as elements of the additive group $\re/\ints$. Recall that the direct product of $d$ copies of $\re/\ints$ is a \textsl{compact torus}\index{compact torus}. Suppose that $t$ is an element of a compact torus $T$. We say that $t$ is a \textsl{generator}\index{torus generator} if the smallest closed subgroup of $T$ containing $t$ is $T$ itself.

\begin{theorem}[Kronecker]
\label{kronecker}
Let $(t_1,\dots,t_r)$ denote an element of $\re^r$, and let $t$ be the image of this point in $T = (\re/\ints)^r$. Then $t$ is a generator of $T$ if and only if $\{1,t_1,\dots,t_r\}$ are linearly independent over $\rats$. \qed
\end{theorem}

\section{$\epsilon$-uniform mixing on $C_{p}$}
In this section we show that $\epsilon$-uniform mixing occurs on cycles of prime length. In particular, we show in these cases that $U(t)$ gets arbitrarily close to the scalar multiple of a complex Hadamard matrix. 

We rely on the viewpoint of cyclic association schemes. Let $p$ denote an odd prime and consider the cycle $C_p$. Let $d = \lfloor \frac{p}{2} \rfloor$. For $0 \leq r \leq d$ we define the following adjacency matrices.
\[
[A_r]_{j,k} = \begin{cases} 1 \; \; \text{if} \; \; j-k \in \{r, -r\} \pmod p \\ 0 \; \; \text{otherwise}.  \end{cases}
\]
Note that $A_0$ is the $p \times p$ identity matrix, and $A_1$ is the adjacency matrix of the cycle $C_p$. Let $\mathcal{A} = \{A_0, \dots, A_d\}$. The set of matrices $\mathcal{A} = \{A_0, \dots, A_d\}$ form the \textsl{cyclic association scheme} of order $p$.

It is also convenient to consider the underlying permutation matrix $C$ that is the adjacency matrix of a directed cycle. We index the rows and columns of $C$ with elements of $\ints_p$ such that
\[
C_{j,k} = \begin{cases} 1 \; \; \text{if} \; \; j-k \equiv 1 \pmod p \\ 0 \; \; \text{otherwise}.  \end{cases}
\]

Let $\mathcal{A} = \{A_0, \dots, A_d\}$ denote the cyclic association scheme on $p$ vertices. Since $p$ is an odd prime, we see that $\mathcal{A}$ has $d = (p-1)/2$ classes.  We further let $\{E_0, \dots, E_d\}$ denote the spectral idempotents of $\mathcal{A}$. We assume that these idempotents have been ordered such that $E_0 = \frac{1}{p}J$ and 
\[
E_r = \sum_{j = 0}^{n-1}\left( \omega^{jr} + \omega^{-jr} \right) C^j. 
\]
for $1 \leq r < \lfloor n/2 \rfloor$.

To begin our new work, we define
\begin{equation}
\label{Fp_equation}
F = \sum_{r=0}^d \omega^{r^2} E_r,
\end{equation}
where $\omega = e^{i2\pi/p}$.

\begin{lemma}
The matrix $F$ is a flat unitary matrix.\qed
\end{lemma}
\begin{proof}
First we verify that $F$ is unitary. We do this by a direct computation. It is convenient to recall that 
\[
E_r^2 = E_r \quad \text{and} \quad E_r E_j = 0 \; \; \text{if} \; \; r \neq j .
\]
Using these observations we see that
\[
F F^{*} = \left(\sum_{r=0}^d \omega^{r^2} E_r \right) \left(\sum_{r=0}^d \omega^{-r^2} E_r \right) = \sum_{r = 0}^d E_r = I.
\]
Next we use the discrete Fourier transform $\Theta$ to show that $F$ is flat. Note that for an arbitrary matrix $M$ in $\cx[\mathcal{A}]$ there exists a unique polynomial $p(x)$ of degree at most $p-1$ in $\cx[x]$ such that $M = p(C)$. Further note that $\Theta$ is defined such that 
\[
\Theta(M) = \sum_{j =0}^{n-1} p(\omega^j) C^j. 
\]
Since $\Theta$ is linear, we have
\begin{align*}
\Theta(F) \Theta(F^{*}) &= \left( \sum_{r=0}^d \omega^{r^2} \Theta(E_r) \right) \left( \sum_{r=0}^d \omega^{-r^2} \Theta(E_r) \right) \\
&=  \left( \sum_{r=0}^d \omega^{r^2} A_r \right) \left( \sum_{r=0}^d \omega^{-r^2} A_r \right) \\
&= \left( \sum_{j=0}^{p-1} \omega^{j^2} C^j \right) \left( \sum_{j=0}^{p-1} \omega^{-j^2}C^j \right) \\
&= \sum_{k=0}^{p-1} \left(\sum_{j=0}^{p-1} w^{j^2 - (k-j)^2} \right) C^{k} \\
&= \sum_{k=0}^{p-1} \omega^{-k^2} \left(\sum_{j=0}^{p-1} w^{jk} \right) C^{k} \\
&= pC^0 = pI.
\end{align*}
By a well-known property of the discrete Fourier transform, we know that $\Theta(F) \Theta(F^*) = p I$ implies that
\[
F \circ F^{*} = \frac{1}{p}J.
\]
Therefore $F$ is flat.\qed
\end{proof}

Our goal is to show that $U(t)$ gets arbitrarily close to a complex scalar multiple of $F$ as $t$ ranges over all real numbers. Since $F$ is a flat matrix, achieving this goal implies that $C_p$ admits $\epsilon$-uniform mixing. The proof of this result relies heavily on Kronecker's Theorem.

\begin{theorem}
\label{CpepsilonUM}
The odd prime cycle $C_p$ admits $\epsilon$-uniform mixing.
\end{theorem}

\begin{proof}
Let $U'(t)$ denote the scaled transition matrix given by
\begin{equation}
\label{scaled_transition}
U'(t) = e^{-2it} U(t) = E_0 + e^{(\theta_1 - 2)it}E_1 + \dots + e^{(\theta_d -2)it} E_d.
\end{equation}
Note that $U'(t)$ is a unitary matrix, and $U'(t)$ is flat if and only if $U(t)$ is flat. Let $\epsilon$ denote a positive real number. We proceed by showing that there exists some time $t$ such that 
\[
|| U'(t) - F || < \frac{ \epsilon}{2}.
\]
We consider $U'(t)$ at times that are an integer multiple of $2 \pi/p$. For any $s$ in $\ints$, we see that Equation \ref{scaled_transition} becomes
\[
U'(2 s \pi/p) = \sum_{r=0}^{d} e^{2s(\theta_r - 2) \pi i/p} E_r.
\]
In terms of $e$, we express Equation \ref{Fp_equation} as
\[
F = \sum_{r=0}^{d} e^{2r^2 \pi i/p} E_r.
\]
Our goal is to find a time $t$ such that the coordinates of $F$ and $U'(t)$ are close to the same value. In terms of the exponents of these coefficients, this is equivalent to finding some integer $s$ such that $\frac{1}{p}r^2 \approx \frac{1}{p}(\theta_r - 2)s$ in $(\re/\ints)$ for $0 \leq r \leq d$. For two elements $x$ and $y$ in $\re/\ints$, we define the distance $|x - y|_{\re/\ints}$ to be
\[
|x - y|_{\re/\ints} = \inf_{k \in \ints}|x-y-k|,
\]
where the norm on the right hand side of the definition is the absolute value of $x-y-k$ considered as a real number.

From Theorem \ref{prime_independence}, we know that $1, \theta_1, \dots, \theta_{d-1}$ are linearly independent over $\rats$, and consequently $\left\{1, \frac{1}{p}(\theta_{1} -2), \dots, \frac{1}{p}(\theta_{d-1} -2)\right\}$ is linearly independent over the rationals.

By Kronecker's Theorem (Theorem \ref{kronecker}), we see that 
\[
D = \left\{ \left(\frac{1}{p}(\theta_{1} -2)s, \dots, \frac{1}{p}(\theta_{d-1} -2)s \right) : s \in \ints \right\}
\]
is dense in $(\re/\ints)^{d-1}$. 

Therefore for any $\delta > 0$, we can find some $s$ in $\ints$ such that
\begin{equation}
\label{pcyc_exponents}
\left| \frac{1}{p} (\theta_r - 2)s - \frac{r^2}{p} \right|_{\re/\ints} < \delta
\end{equation}
in $(\re/\ints)$ for all $1 \leq r \leq d-1$. It remains to consider the coordinates of $U'(t)$ and $F$ with respect to $E_d$.  Recall that for a cyclic association scheme we have
\[
    \theta_d = -1 - \theta_1 - \theta_2 \dots - \theta_{d-1}.
\]
We can use this to derive an expression for the $d$-th coordinate of $U'(t)$ in terms of the first $d-1$ coordinates.
\begin{align*}
	\frac{1}{p}(\theta_d -2)s &= \frac{1}{p}\left(-3 - \sum_{r=1}^{d-1} \theta_r\right)s \\
	&= - \frac{1}{p}(2(d-1) + 3)s - \sum_{r=1}^{d-1}\frac{1}{p}(\theta_r - 2)s \\
	\label{coordinate_simp2_equation} &= -s - \sum_{r=1}^{d-1}\frac{1}{p}(\theta_r - 2)s.
\end{align*}
Now working in $\re/\ints$, we see that the exponent of the $d$-th coordinate of $U'(t)-F$ is
\begin{align}
\frac{1}{p}(\theta_d -2)s - \frac{1}{p}d^2 &=  -s - \sum_{r=1}^{d-1}\frac{1}{p}(\theta_r - 2)s - \frac{1}{p}d^2 \\
&=  \sum_{r=1}^{d-1}\left( \frac{1}{p}(\theta_r - 2)s - \frac{1}{p}r^2 \right) + \frac{1}{p} \sum_{r=0}^d r^2 .
\end{align}
Note that
\begin{align*}
   \frac{1}{p} \sum_{r=1}^{d} r^2 &= \frac{d(d+1)(2d+1)}{6p} = \frac{(p-1)(p+1)}{24}
\end{align*}
Since $p$ is an odd prime, we know that both $p-1$ and $p+1$ are even, and exactly one of those values is divisible by 4. Therefore $(p-1)(p+1)$ is divisible by 8. Since we are assuming that $p \neq 3$, we also know that  $p-1$ or $p+1$ is divisible by 3. It follows that 
\[
\frac{(p-1)(p+1)}{24} \in \ints.
\]
must be an integer. We simplify Equation \ref{coordinate_simp2_equation} in $\re/\ints$ to
\begin{equation*}
\frac{1}{p}(\theta_d -2)s - \frac{1}{p}d^2 =  \sum_{r=1}^{d-1}\left( \frac{1}{p}(\theta_r - 2)s - \frac{1}{p}r^2 \right).
\end{equation*}
Now we use this expression and Inequality \ref{pcyc_exponents} to bound the coefficient $d$-th coordinate of $U'(t)-F$ in terms of $\delta$ as follows:
\begin{align*}
\left| \frac{1}{p}(\theta_d -2)s - \frac{1}{p}d^2 \right|_{\re/\ints} & \leq  \sum_{r=1}^{d-1}\left| \frac{1}{p}(\theta_r - 2)s - \frac{1}{p}r^2 \right|_{\re/\ints} < (d-1)\delta.
\end{align*}
This implies that for any $\epsilon > 0$, we can find a sufficiently small $\delta$ such that 
\[
\left|\left|U'(2s \pi/p) - F \right|\right| < \frac{\epsilon}{2}.
\]
It can be shown that if $A$ and $B$ are symmetric $n \times n$ complex matrices, such that 
\[
|| A - B || \leq \gamma,
\]
for some positive real number $\gamma$. Then
\[
|| A \circ A^* - B \circ B^* || \leq 2 \gamma.
\]
Therefore it follows that
\[
|| U'(2s \pi/p) \circ U'(2s \pi/p)^{*} - \frac{1}{n}J || < \epsilon.
\]
Finally, we note that $U'(2s \pi/p) \circ U'(2s \pi/p)^{*} = U(2s \pi/p) \circ U(2s \pi/p)^{*}$, which proves that $\epsilon$-uniform mixing occurs on $C_p$.\qed
\end{proof}

\section{Future work}
Our results strongly supports the conjecture of Ahmadi et al.~\cite{ahmadi03} that no cycle $C_n$, except for $C_3$ and $C_4$, admits uniform mixing. In this paper we proved that no prime cycles, except for $C_3$, admit uniform mixing. Thus it would be desirable to relate the mixing properties of 
$C_{pm}$ to mixing properties of $C_m$ for any prime $p$ and integer $m$. This could lead to a complete proof of Ahmadi et al.'s conjecture.

Another area for future work is finding additional examples of graphs in association schemes that admit uniform mixing. Recent work by Chan shows that such graphs exist in certain schemes derived from $n$-cubes \cite{chan13}.

In addition to those problems, there are still basic questions concerning uniform mixing that remain unanswered. For example, it is unknown whether a graph that admits uniform mixing is necessarily regular.

\bibliographystyle{plain}
\bibliography{IUM}

\end{document}